\documentclass[reqno]{amsart}
\usepackage{amsmath, amsthm, amscd, amssymb, amsfonts, amsbsy}
\usepackage{latexsym, color, enumerate}
\usepackage{mathrsfs}
\usepackage{pxfonts}
\usepackage{bbm}
\usepackage[dvipsnames]{xcolor}
\usepackage{url}
\usepackage{enumerate}
\usepackage{todonotes}
\usepackage{marginnote}

\theoremstyle{plain}
\newtheorem{theorem}[equation]{Theorem}
\newtheorem{lemma}[equation]{Lemma}
\newtheorem{corollary}[equation]{Corollary}
\newtheorem{proposition}[equation]{Proposition}

\theoremstyle{definition}

\theoremstyle{remark}
\newtheorem{remark}[equation]{Remark}

\newcommand{\jj}{\mathrm{j}}

\numberwithin{equation}{section}

\def\dVert{\,\,\text{--}\kern-.46em\|}

\begin{document}

\title[Hypercontractivity of Poisson Semigroups with Orthogonal Polynomial Eigenfunctions]
{Hypercontractivity of Poisson Semigroups with Orthogonal Polynomial Eigenfunctions}

\author[M. Hormozi]{Mahdi Hormozi}
\address[M. Hormozi]{Beijing Institute of Mathematical Sciences and Applications (BIMSA), Beijing 101408, China}
\email{hormozi@bimsa.cn}

\author[J.-X. Zhu]{Jie-Xiang Zhu}
\address[J.-X. Zhu]{Department of Mathematics, Shanghai Normal University, Shanghai 200234, China}
\email{15110840006@fudan.edu.cn}



\begin{abstract}
For any $1 < p < q < \infty$, we investigate fixed-time hypercontractive bounds from $L^p$ to $L^q$ of Poisson semigroups associated with the Ornstein--Uhlenbeck, Laguerre and Jacobi operators. We prove that, in the Ornstein--Uhlenbeck and Laguerre cases, the Poisson semigroups fail to be $L^p \to L^q$ bounded for any fixed $t > 0$. In contrast, for Jacobi operators with $\alpha, \beta \ge -1/2$, the associated Poisson semigroups are ultracontractive, namely bounded from $L^1$ to $L^\infty$.

More generally, we study Bernstein subordinations of these semigroups and show that fixed-time hypercontractivity is not stable under subordination. The analysis relies on quantitative $L^q$-estimates for the corresponding orthogonal polynomial eigenfunctions, together with a bilinear test with the exponential family.
\end{abstract}

\maketitle

\section{Introduction}

The Ornstein--Uhlenbeck, Laguerre and Jacobi semigroups are three classical families of symmetric diffusion semigroups with orthogonal polynomial eigenfunctions. In dimension one, they are in fact the only such cases up to affine transformations; see for instance \cite{Urb, BGL, UrbWorkshop}.

For the Gaussian probability measure $d \gamma_d(x) = \pi^{-d/2} e^{-|x|^2}\, dx$ on $\mathbb R^d$, the Ornstein--Uhlenbeck operator (in the normalization of \cite{Urb})
$$
L_{\rm OU} := \frac12\Delta-\langle x,\nabla\rangle
$$
generates the Ornstein--Uhlenbeck semigroup $T_s = e^{sL_{\rm OU}}$ for $s \ge 0$. The celebrated Nelson--Gross theorem \cite{Nelson,Gross} yields the sharp hypercontractive estimate: for $1<p<\infty$ and $s>0$,
\begin{equation} \label{eq:intro-hypercontractivity}
\| T_s\|_{L^p(\gamma_d)\to L^{q(s)}(\gamma_d)}\le 1,
\qquad
q(s) := 1 + (p-1) e^{2s}.
\end{equation}
This regularity property is known to be equivalent to the logarithmic Sobolev inequality. More generally, a logarithmic Sobolev inequality follows from a suitable curvature condition $CD(\rho,\infty)$ with $\rho>0$; see e.g. \cite[Sect. 5.7]{BGL}. From this point of view, the Laguerre operator $L_\alpha$ on $(0,\infty)$ with $\alpha\ge -1/2$ and the Jacobi operator $L_{\alpha,\beta}$ on $(-1,1)$ with $\alpha,\beta\ge -1/2$ also satisfy a positive curvature condition, thus their associated Markov semigroups admit hypercontractive estimates of the form \eqref{eq:intro-hypercontractivity} with respect to their respective invariant probability measures.

Bernstein subordination is a standard procedure for constructing new semigroups from a given one via functional calculus. A function $f: [0,\infty)\to[0,\infty)$ is called a Bernstein function if $f\in C^\infty((0,\infty))$ and, for any  $\lambda > 0$ and $k \in \mathbb N_0$,
$$
(-1)^k f^{(k+1)}(\lambda) \ge 0.
$$
Equivalently, by the L\'evy--Khintchine formula, there exist $a,b\ge 0$ and a Borel measure $\nu$ on
$(0,\infty)$ satisfying
$$
\int_{(0,\infty)}(1\wedge s) \, d \nu (s) <\infty
$$
such that
\begin{equation}\label{eq:bernstein-LK}
f(\lambda)=a+b\lambda+\int_{(0,\infty)}\bigl(1- e^{-s\lambda}\bigr)\, d \nu(s).
\end{equation}
We refer to $(a,b,\nu)$ as the L\'evy--Khintchine triplet of $f$; see \cite[Definition~3.1]{SW} and \cite{SSV}. Throughout the paper we exclude the trivial case $f\equiv0$; in particular, $\liminf_{\lambda \to \infty} f(\lambda) > 0$. Let $(T_s)_{s\ge 0}$ be a symmetric Markov semigroup with generator $L$ and invariant probability measure $\mu$, and let $f$ be a Bernstein function given by \eqref{eq:bernstein-LK}. Bochner subordination yields the semigroup
$S_t^{\,f} : = e^{-t f(-L)}$; see, e.g., \cite{SW,GM}.  The particular choice $f(\lambda)=\sqrt{\lambda}$ leads to the Poisson semigroup
$$
P_t:= e^{-t(-L)^{1/2}}, \quad t>0,
$$
which is a fundamental tool in harmonic analysis. The subordinate semigroup $(S_t^{\, f})_{t\ge 0}$ is sub-Markov, in the sense that $S_t^{\, f} \mathbf 1 = e^{-at} \mathbf 1 \le \mathbf 1 $, and it is Markov if and only if $a = 0$ in \eqref{eq:bernstein-LK}. Moreover, subordination preserves contractivity on $L^p(\mu)$ for all $1\le p\le\infty$; however, on a probability space these properties do not, in general, imply any off-diagonal $L^p \to L^q$ bound with $1 < p < q < \infty$. For fixed $t > 0$, we therefore ask whether the operator $P_t$, and more generally $S_t^{\, f}$, is bounded from $L^p(\mu)$ to $L^q(\mu)$. In this paper we investigate this question for the Ornstein--Uhlenbeck, Laguerre and Jacobi semigroups.

Throughout, $C$ denotes a positive constant whose value may change from line to line; subscripts indicate dependence. We write $A\lesssim B$ if $A\le CB$ with a constant independent of the main asymptotic parameter, and $A\simeq B$ if $A\lesssim B\lesssim A$. When the underlying measure space is clear, we use $\| \cdot \|_{p\to q}$ for operator norms and $\langle g, h\rangle_\mu := \int g h \, d\mu$ for the duality pairing; we also set $a \wedge b = \min\{a, b\}$. We shall refer to the Poisson semigroups associated with the Ornstein--Uhlenbeck, Laguerre and Jacobi operators as the Poisson--Hermite, Poisson--Laguerre and Poisson--Jacobi semigroups, respectively, and we write the latter two as $(P_t^\alpha)_{t\ge 0}$ and $(P_t^{\alpha, \beta})_{t\ge 0}$ to make the parameter dependence explicit. We now state our main results on the hypercontractivity of these Poisson semigroups.

\begin{theorem}\label{thm:intro-main1}
Here and below, all $L^p$--norms and operator norms are taken on the natural invariant probability spaces of the corresponding semigroups.
\begin{enumerate}[(i)]
\item Let $(P_t)_{t\ge0}$ be the Poisson--Hermite semigroup. Then, for any $1<p<q<\infty$ and $t>0$,
$$
\| P_t \|_{p \to q} = +\infty.
$$

\item Let $(P_t^\alpha)_{t\ge0}$ be the Poisson--Laguerre semigroup with parameter $\alpha > -1$. Then, for any $1<p<q<\infty$ and  $t>0$,
$$
\| P_t^{\alpha} \|_{p \to q} = +\infty.
$$

\item Let $(P_t^{\alpha, \beta})_{t\ge 0}$ be the Poisson--Jacobi semigroup with parameters $\alpha, \beta \ge -1/2$. Then, for any $1 \le p \le q \le \infty$ and  $t > 0$,
$$
\big\| P_t^{\alpha, \beta} \big\|_{p \to q} \lesssim (1 \wedge t)^{-2(p^{-1} - q^{-1})(\max\{ \alpha, \beta \} +1 )}.
$$
\end{enumerate}
\end{theorem}

Theorem~\ref{thm:intro-main1} shows that no fixed-time $L^p \to L^q$ improvement for the Poisson semigroups holds in the Gaussian and Laguerre settings. The reason is spectral: for any fixed $q>2$ and $p < q$, the $L^q/L^p$-norm ratio of the polynomial eigenfunctions grows exponentially in the degree, while the Poisson multiplier decays only as $e^{-t\sqrt{n}}$. In contrast, the Jacobi semigroup with parameters $\alpha, \beta \ge -1/2$ satisfies the ultracontractive estimate
$$
\big\| e^{s L_{\alpha, \beta}}  \big\|_{1 \to \infty} \lesssim (1 \wedge s)^{- (\max \{\alpha, \beta  \}+1)}
$$
which, being preserved under Poisson subordination, implies that the associated Poisson--Jacobi semigroup is also ultracontractive. Theorem~\ref{thm:intro-main1} also shows that the Jacobi hypercontractive/ultracontractive estimates are not stable under the standard
Jacobi-to-Laguerre and ultraspherical-to-Hermite limit relations; see \cite[Appendix~B]{UrbWorkshop} and
Section~\ref{sec:unified-characterization}.

Finally, we turn to the general subordinated semigroups $S_t^{\,f}= e^{-t f(-L)}$ in the three settings considered above.

\begin{theorem}\label{thm:intro-main2}
Let $f$ be a Bernstein function with L\'evy--Khintchine triplet $(a, b, \nu)$, and let $S_t^{\, f}$ denote the subordination of the symmetric Markov semigroup $(e^{s L})_{s\ge0}$.
\begin{enumerate}[(i)]
\item Let $L = L_{\rm{OU}}$ be the Ornstein--Uhlenbeck operator. Then, for any $1< p < \infty$ and $t > 0$,
\begin{align*}
\big\|  S_t^{\, f} \big\|_{p \to q} = \begin {cases}
e^{-at} & \text {if \, $q \le 1 + (p-1) e^{2bt}$,}\\
+\infty      & \text{if \, $q > 1+ (p-1)e^{2bt}$.}
\end {cases}
\end{align*}

\item Let $L = L_{\alpha}$ be the Laguerre operator with parameter $\alpha \ge -1/2$. Then, for any $1< p < \infty$ and $t > 0$,
\begin{align*}
\big\|  S_t^{\, f} \big\|_{p \to q} = \begin {cases}
e^{-at} & \text {if \, $q \le 1 + (p-1) e^{bt}$,}\\
+\infty      & \text{if \, $q > 1 + (p-1) e^{bt}$.}
\end {cases}
\end{align*}

\item Let $L = L_{\alpha, \beta}$ be the Jacobi operator with parameters $\alpha, \beta \ge -1/2$. Assume there exists $\theta \in (0, 1]$ such that
$$
\liminf_{\lambda \to \infty} \lambda^{-\theta} f(\lambda) > 0.
$$
Then, for any $1 \le p \le q \le \infty$ and $t > 0$, we have
$$
\big\| S_t^{\, f} \big\|_{p \to q}
\lesssim (1 \wedge t)^{-\frac{(p^{-1} - q^{-1})(\max\{ \alpha, \beta \} + 1)}{\theta}}.
$$
\end{enumerate}
\end{theorem}

For the Ornstein--Uhlenbeck and Laguerre settings, a nutural dichotomy emerges in the fixed-time $L^p \to L^q$ behaviors of the subordinated semigroups. In the boundedness regime, the $L^p \to L^q$ bounds follow from hypercontractive estimates for the underlying diffusion semigroups. In the blow-up regime, however, the spectral obstruction used earlier is not sharp enough, especially when $1 < p < 2$; we therefore employ a bilinear test with the exponential family $g_\tau(x) = e^{\tau x}$, inspired by \cite{Nelson}. In particular, this approach yields an explicit $L^p \to L^q$ boundedness criterion for non-negative Laguerre multipliers (an ``exponential obstruction''); see Section~\ref{sec:blow-up-ou-lauguerre} below. For the Jacobi setting, the required fixed-time bounds follow from a super Poincar\'e inequality, which is equivalent to ultracontractivity in this framework.

The paper is organized as follows. Section~\ref{sec:settings} collects preliminaries and auxiliary results used throughout the paper. Section~\ref{sec:proof-of-1} contains the proof of Theorem~\ref{thm:intro-main1} and discusses the degeneration of hypercontractive bounds for Poisson semigroups along Jacobi limit transitions. Section~\ref{sec:OU-bernstein-drift} contains the proof of Theorem~\ref{thm:intro-main2}.

\section{Preliminaries}\label{sec:settings}

\subsection{Facts on the Ornstein--Uhlenbeck, Laguerre and Jacobi semigroups} \label{sec:facts}

This section collects basic facts on the Ornstein--Uhlenbeck, Laguerre and Jacobi semigroups used throughout the paper; see, e.g., \cite{Urb, BGL, UrbWorkshop} for details.

Let $d \gamma_d(x) = \pi^{-d/2} e^{-|x|^2} \, dx$ be the Gaussian measure on $\mathbb R^d$. The Ornstein--Uhlenbeck generator is
$$
L_{\rm OU} := \frac12\Delta-\langle x,\nabla\rangle,
$$
with invariant measure $\gamma_d$; (see \cite[(2.14)--(2.16)]{Urb} for the closely related normalizations). In dimension one, the eigenfunctions of the Ornstein--Uhlenbeck operator are given by the Hermite polynomials. We denote by $(H_n)_{n\ge0}$ the Hermite polynomials defined by the generating function
\begin{equation} \label{eq:OU-generating}
\exp(2xt - t^2) = \sum_{n = 0}^{\infty} \frac{H_n(x)}{n!} \, t^n,
\end{equation}
and by $(h_n)_{n\ge0}$ their $L^2(\gamma_1)$--normalized versions. In higher dimensions, the eigenfunctions of $L_{\rm OU}$ are given by tensor
products of the one-dimensional Hermite polynomials. More precisely, the family
\[
h_{\bf k}(x) = \prod_{j=1}^d h_{k_j}(x_j),\qquad {\bf k} = (k_1, \ldots, k_d) \in\mathbb N_0^d,
\]
forms an orthonormal basis of $L^2(\gamma_d)$, and
\[
L_{\rm OU} h_{\bf k} = -|{\bf k}|\, h_{\bf k}, \qquad |{\bf k}|: = \sum_{j=1}^d k_j.
\]
In particular, the spectrum of $L_{\rm OU}$ is $\mathbb N_0$. We shall later use $L^q(\gamma_1)$ bounds for Hermite polynomials; see Section~\ref{sec:blow-up}.

Fix $\alpha>-1$ and let $\mu_\alpha$ be the normalized Gamma measure on $(0,\infty)$,
\[
d \mu_\alpha (x) : = \frac{1}{\Gamma(\alpha+1)}\, x^\alpha e^{-x}\, dx.
\]
The Laguerre operator in the normalization of \cite{Urb} is
\[
L_\alpha f(x) := x f''(x)+ (\alpha+1-x)f'(x),\qquad x>0,
\]
which is symmetric on $L^2(\mu_{\alpha})$ and generates the Laguerre semigroup $(e^{s L_\alpha})_{s\ge0}$. The Laguerre polynomials $(L_n^\alpha)_{n\ge0}$ may be characterized by the generating function
\begin{equation} \label{eq:Laguerre-generating}
\sum_{n=0}^\infty L_n^\alpha(x)\,r^n=(1-r)^{-\alpha-1}\exp\!\Bigl(-\frac{r}{1-r}\,x\Bigr),
\qquad |r|<1 .
\end{equation}
Let $(\ell_n^\alpha)_{n\ge0}$ denote the $L^2(\mu_\alpha)$--normalized Laguerre polynomials, which form an orthonormal basis of
$L^2(\mu_\alpha)$ and satisfy $L_\alpha \ell_n^\alpha=-n \, \ell_n^\alpha $ for $n \in \mathbb N_0$.

From the point of view of the curvature-dimension condition, if $\alpha \ge -1/2$, the Laguerre operator $L_\alpha$ satisfies $CD(1/2,\infty)$; see e.g. \cite[Sect. 2.7]{BGL}. Consequently, the Laguerre semigroup is hypercontractive (see \cite[Chap.~5]{BGL}): for $1< p < \infty$ and $s > 0$,
\begin{equation} \label{eq:pre-hypercontractivity}
\big\| e^{s L_\alpha} \big\|_{p \to q(s)}\le 1,
\qquad q(s) := 1 + (p-1)e^{s}.
\end{equation}

Fix $\alpha,\beta>-1$ and let $\mu_{\alpha,\beta}$ be the Jacobi probability measure on $(-1,1)$,
\[
d \mu_{\alpha, \beta}(x) : =
Z_{\alpha,\beta}^{-1} \, (1-x)^{\alpha}(1+x)^{\beta}\mathbf 1_{(-1,1)}(x)\, dx,
\qquad
Z_{\alpha,\beta} := 2^{\alpha+\beta+1}B(\alpha+1,\beta+1).
\]
The Jacobi operator is
\[
L_{\alpha, \beta} f (x) := (1 - x^2) f''(x) - \big[(\alpha + \beta + 2)x + \alpha - \beta   \big] f'(x),\qquad x \in (-1, 1),
\]
which is symmetric on $L^2(\mu_{\alpha, \beta})$ and generates the Jacobi semigroup $(e^{s L_{\alpha, \beta}})_{s\ge0}$. The eigenfunctions of $L_{\alpha,\beta}$ are the Jacobi polynomials. We write $\big( J_n^{(\alpha,\beta)} \big)_{n\ge0}$ for the classical Jacobi polynomials and $\big( \jj_n^{(\alpha,\beta)} \big)_{n\ge0}$ for their $L^2(\mu_{\alpha,\beta})$--normalized versions, which form an orthonormal basis of $L^2(\mu_{\alpha,\beta})$ and satisfy
\[
L_{\alpha,\beta} \, \jj_n^{(\alpha,\beta)}=- \lambda_n^{(\alpha, \beta)}\, \jj_n^{(\alpha,\beta)}, \qquad
\lambda_n^{(\alpha, \beta)} := n(n+\alpha+\beta+1),\qquad n \in \mathbb N_0.
\]
We refer to \cite{Szego} for background on Jacobi polynomials. The Jacobi heat kernel, namely the transition kernel of $e^{s L_{\alpha, \beta}}$ with respect to $\mu_{\alpha, \beta}$, is then given by
\[
G_s^{(\alpha,\beta)} (x, y) := \sum_{n = 0}^{\infty} e^{-s \lambda_n^{(\alpha, \beta)}} \, \jj_n^{(\alpha,\beta)}(x) \, \jj_n^{(\alpha,\beta)}(y), \qquad x, y \in (-1, 1), \,  s > 0.
\]
If $\alpha, \beta \ge -1/2$, bounds on $G_s^{(\alpha,\beta)}$ due to \cite{NowSj} imply the following ultracontractive estimate:
\begin{equation} \label{eq:jacobi-ultra}
\big\| e^{s L_{\alpha, \beta}}  \big\|_{1 \to \infty} \lesssim (1 \wedge s)^{- (\max \{\alpha, \beta  \}+1)}, \qquad s > 0,
\end{equation}
which alternatively follows from a suitable Sobolev inequality as developed in \cite{B96}.

\subsection{Spectral obstructions for hypercontractivity}\label{sec:front-results}

Some of our $L^p\to L^q$ blow-up results for Bernstein subordinations rely on the following simple lemma.

\begin{lemma}\label{lem:unified-eigen-test}
Let $(X,\mu)$ be a probability space and let $L$ be a self-adjoint operator on $L^2(\mu)$ such that $-L\ge0$.
Let $1<p<q<\infty$. Assume that $-L$ has discrete spectrum with eigenvalues $(\lambda_n)_{n\ge0}$ listed in the increasing order counting multiplicities, with
$\lim_{n \to \infty} \lambda_n = + \infty$, and eigenfunctions $(\varphi_n)_{n\ge0}\subset L^p(\mu)\cap L^q(\mu)$, so that $L\varphi_n=-\lambda_n\varphi_n$. Let $f$ be a Bernstein function.
For $t>0$ define the subordinated semigroup $S_t^{\, f}:=e^{-t f(-L)}$, so that $S_t^{\, f} \varphi_n = e^{-t f(\lambda_n)} \, \varphi_n$. Then,
$$
\big\| S_t^{\, f} \big\|_{p \to q} \ge \sup_{n\ge0} e^{-tf(\lambda_n)} \, \frac{\| \varphi_n \|_{L^q(\mu)}}{\| \varphi_n \|_{L^p(\mu)}}.
$$
\end{lemma}

\begin{proof}
For each $n$, we test the operator norm on $g = \varphi_n/\| \varphi_n \|_{L^p(\mu)}$,  so that $\|g \|_{L^p(\mu)}=1$. Hence
$$
\big\| S_t^{\, f} \big\|_{p \to q}
=\sup_{\|g\|_{L^p(\mu)}=1} \big\| S_t^{\, f} g \big\|_{L^q(\mu)} \ge \frac{\big\| S_t^{\, f} \varphi_n \big\|_{L^q(\mu)}}{\| \varphi_n \|_{L^p(\mu)}}
= e^{-t f(\lambda_n)} \frac{\| \varphi_n \|_{L^q(\mu)}}{\| \varphi_n \|_{L^p(\mu)}}.
$$
Taking the supremum over $n\ge0$ yields the claim.
\end{proof}

As a consequence of Lemma~\ref{lem:unified-eigen-test}, we obtain the following $L^p\to L^q$ blow-up criterion for Bernstein subordinations. This blow-up occurs because the decay of the spectral multiplier cannot compete with the growth of the $L^q/L^p$-norm ratio of the eigenfunctions.

\begin{theorem}\label{thm:SW-spectral-obstruction}
Assume the setting of Lemma~\ref{lem:unified-eigen-test}. Define
\[
\Theta_{p, q, f} : =\limsup_{n \to \infty} \frac{\log \big(\| \varphi_n \|_{L^q(\mu)}/ \| \varphi_n \|_{L^p(\mu)} \big)}{f(\lambda_n)} \in [0, +\infty].
\]
\begin{enumerate}[(i)]
\item If $\Theta_{p,q,f} = +\infty$, then $\big\| S_t^{\, f}  \big\|_{p \to q} = +\infty$ for any $t > 0$.
\item If $\Theta_{p,q,f} \in (0, \infty)$ and $\lim_{\lambda \to \infty}f(\lambda) = +\infty$,
then $\big\| S_t^{\, f}  \big\|_{p \to q} = +\infty$ for any $t \in (0, \Theta_{p, q, f})$.
\end{enumerate}
\end{theorem}

\begin{proof}
If $\Theta_{p,q,f} = +\infty$, then for each fixed $t>0$ there exists a subsequence $n_k\to\infty$ such that
$$\log \big(\| \varphi_{n_k} \|_{L^q(\mu)}/ \| \varphi_{n_k} \|_{L^p(\mu)} \big)\ge (t+k)f(\lambda_{n_k}).$$
Hence
$e^{-t f(\lambda_{n_k})} \frac{\| \varphi_{n_k} \|_{L^q(\mu)}}{\| \varphi_{n_k} \|_{L^p(\mu)}} \ge e^{k f(\lambda_{n_k})}\to + \infty$ as $k \to \infty$, since $\liminf_{\lambda \to \infty}f(\lambda) > 0$. If $\Theta_{p, q, f} \in (0, \infty)$ and $t \in (0, \Theta_{p, q, f})$, choose $\varepsilon > 0$ with $t + \varepsilon < \Theta_{p, q, f}$ and take a subsequence $n_k \to\infty$ such that
$$\log \big(\| \varphi_{n_k} \|_{L^q(\mu)}/ \| \varphi_{n_k} \|_{L^p(\mu)} \big)\ge (t + \varepsilon)f(\lambda_{n_k}).$$
Hence $e^{-t f(\lambda_{n_k})} \frac{\| \varphi_{n_k} \|_{L^q(\mu)}}{\| \varphi_{n_k} \|_{L^p(\mu)}} \ge e^{\varepsilon f(\lambda_{n_k})}\to + \infty$ as $k \to \infty$. The conclusion then follows from Lemma~\ref{lem:unified-eigen-test}.
\end{proof}

Taking $f(\lambda)=\sqrt{\lambda}$ in Theorem~\ref{thm:SW-spectral-obstruction}, we obtain the following $L^p\to L^q$ blow-up criterion for the Poisson semigroup.

\begin{corollary} \label{cor:SW-spectral-obstruction-II}
Assume the setting of Lemma~\ref{lem:unified-eigen-test} and let $P_t: = e^{-t(-L)^{1/2}}$. Set
\[
\kappa_{p, q} : = \limsup_{n \to \infty} \frac{\log \big(\| \varphi_n \|_{L^q(\mu)}/ \| \varphi_n \|_{L^p(\mu)} \big)}{\sqrt{\lambda_n}}.
\]
If $\kappa_{p,q}=+\infty$, then $\|  P_t \|_{p \to q} = +\infty$ for any $t > 0$.
\end{corollary}

\subsection{Functional inequalities under subordination} \label{sec:SW-relation}

In this subsection we collect the super- and weak Poincar\'e inequalities transferred in \cite{SW} under Bernstein subordination. These inequalities will be used later in Section~\ref{sec:jacobi-case}.

Let $(T_s)_{s\ge 0}$ be a symmetric Markov semigroup with generator $L$ and invariant probability measure $\mu$. Assume that $L$ is self-adjoint on $L^2(\mu)$. Let $f$ be a Bernstein function, and let $S_t^{\, f} : = e^{-t f(-L)}$ denote the subordinated semigroup. Let $\Phi:L^2(\mu)\to[0,\infty]$ be a functional on $L^2(\mu)$, which is homogeneous of degree $2$ and monotone along the semigroup, i.e., for all $c\in \mathbb R$, $s > 0$ and $u \in L^2(\mu)$,
$$
\Phi(cu)=c^2\Phi(u),\qquad \Phi(T_su)\le \Phi(u).
$$
A typical example is $\Phi(u) = \|u\|_{L^1(\mu)}^2$.

\begin{proposition}[\text{\cite[Proposition~9]{SW}}]\label{prop:SW-superP}
Assume that $L$ satisfies the super-Poincar\'e inequality
$$
\|u\|_{L^2(\mu)}^2
\le r\,\langle (-L) u,u\rangle_{L^2(\mu)}+\beta(r)\,\Phi(u),
\qquad r>0,\ u\in D(L),
$$
for a decreasing function $\beta:(0,\infty)\to(0,\infty]$.
Let $f$ be a Bernstein function, then $f(-L)$ satisfies
$$
\|u\|_{L^2(\mu)}^2
\le r\,\langle f(-L)u,u\rangle_{L^2(\mu)}+\beta_f(r)\,\Phi(u),
\qquad r>0,\ u\in D(f(-L)),
$$
with the transformed function
$$
\beta_f(r) : = 4\,\beta\!\left(\frac{1}{2\,f^{-1}(2/r)}\right).
$$
\end{proposition}

\begin{proposition}[\text{\cite[Proposition~10]{SW}}]\label{prop:SW-weakP}
Assume that $L$ satisfies the weak Poincar\'e inequality
$$
\|u\|_{L^2(\mu)}^2
\le \alpha(r)\,\langle (-L) u,u\rangle_{L^2(\mu)}+r\,\Phi(u),
\qquad r>0,\ u\in D(L),
$$
for a decreasing function $\alpha:(0,\infty)\to(0,\infty)$.
Let $f$ be a Bernstein function, then $f(-L)$ satisfies
$$
\|u\|_{L^2(\mu)}^2
\le \alpha_f(r)\,\langle f(-L)u,u\rangle_{L^2(\mu)}+r\,\Phi(u),
\qquad r>0,\ u\in D(f(-L)),
$$
with the transformed function
$$
\alpha_f(r) := \frac{2}{\,f\!\left(\frac{1}{2\,\alpha(r/4)}\right)}.
$$
\end{proposition}

\begin{remark}
If $f(\lambda)=\sqrt{\lambda}$, then $f^{-1}(u)=u^2$, so
$$
\beta_{1/2}(r)=4\,\beta\!\left(\frac{r^2}{8}\right),
\qquad
\alpha_{1/2}(r)=2\sqrt{2\,\alpha(r/4)}.
$$
Thus Poisson subordination reshapes the associated rate functions by the deformation $r\mapsto r^2$
(super-Poincar\'e) and a square-root renormalization (weak Poincar\'e).
\end{remark}

By choosing suitable rate functions and functionals $\Phi$, the super- and weak Poincar\'e inequalities above yield a variety of useful consequences, including ultracontractivity and concentration inequalities; see, e.g., \cite{W05}.

\section{Proof of Theorem \ref{thm:intro-main1}} \label{sec:proof-of-1}
\subsection{Ornstein--Uhlenbeck and Laguerre: blow-up} \label{sec:blow-up}

In this section we prove the first two assertions. We start with the Poisson--Hermite semigroup. To this end, we establish the following $L^q$--estimate for the Hermite polynomials $(h_n)_{n \ge 0}$.

\begin{proposition}\label{prop:Hermite-exp-growth}
Fix $q>2$. Then for every $n\ge 1$,
\begin{equation} \label{eq:upper-lower-hermite}
\frac{(q-1)^{n/2}}{(2\pi n)^{1/4}}\exp\!\Bigl(-\frac{1}{24n}\Bigr)
\le
\|h_n\|_{L^q(\gamma_1)}
\le
(q-1)^{n/2}.
\end{equation}
Consequently, for $2 \le p < q$,
\begin{equation} \label{eq:ratio-hermite-I}
\lim_{n \to \infty} \frac{\log\big(\| h_n \|_{L^q(\gamma_1)}/ \| h_n \|_{L^p(\gamma_1)}  \big)}{n} = \frac12 \log\bigg( \frac{q-1}{p-1} \bigg);
\end{equation}
and for $1 < p < 2$,
\begin{equation} \label{eq:ratio-hermite-II}
\liminf_{n \to \infty} \frac{\log\big(\| h_n \|_{L^q(\gamma_1)}/ \| h_n \|_{L^p(\gamma_1)}  \big)}{n} \ge \frac12 \log(q-1).
\end{equation}
\end{proposition}

\begin{proof}
We first prove \eqref{eq:upper-lower-hermite}. Write $q'=\frac{q}{q-1}$. By duality,
\begin{equation} \label{eq:dual-hermite}
\|h_n\|_{L^q(\gamma_1)} = \sup_{\|g\|_{L^{q'}(\gamma_1)}=1} \big| \langle h_n  ,   g  \rangle_{\gamma_1} \big|.
\end{equation}
For $b>0$, we define
$$
g_b(x):=\exp\!\Bigl(bx-\frac{q'b^2}{4}\Bigr).
$$
Since $\int_{\mathbb R}e^{sx}\, d\gamma_1(x)=e^{s^2/4}$, we have $\|g_b\|_{L^{q'}(\gamma_1)}=  1$.

Recall that $h_n=(2^n n!)^{-1/2}H_n$.  Integrating the generating function \eqref{eq:OU-generating} against $e^{bx}\, d \gamma_1$ and comparing coefficients gives
\begin{equation} \label{eq:fourier-expansion-OU}
\int_{\mathbb R} h_n(x)\,e^{bx}\, d\gamma_1(x) = \frac{b^n}{(2^n n!)^{1/2}}\,e^{b^2/4}.
\end{equation}
Therefore, by \eqref{eq:dual-hermite},
\begin{align*}
\|h_n\|_{L^q(\gamma_1)} &\ge \big| \langle h_n  ,   g_b  \rangle_{\gamma_1} \big| = \frac{b^n}{(2^n n!)^{1/2}} \exp\!\Bigl(\frac{b^2}{4}-\frac{q'b^2}{4}\Bigr)\\
& = \frac{b^n}{(2^n n!)^{1/2}}
\exp\!\Bigl(-\frac{b^2}{4(q-1)}\Bigr).
\end{align*}
Optimizing in $b$ gives $b=\sqrt{2n(q-1)}$, hence
$$
\|h_n\|_{L^q(\gamma_1)}
\ge
\frac{(q-1)^{n/2}\,n^{n/2}\, e^{-n/2}}{\sqrt{n!}}.
$$
Using Robbins' refinement of Stirling's formula \cite{Robbins},
$n!\le \sqrt{2\pi}\,n^{n+\frac12} e^{-n} e^{1/(12n)}$, we obtain the stated lower bound.

For the upper bound, set $s:=\frac12\log(q-1)>0$. By hypercontractive estimate \eqref{eq:intro-hypercontractivity} one has $\big\| e^{s L_{\rm OU}} \big\|_{2 \to q}\le 1$. Since $e^{s L_{\rm OU}} h_n = e^{-sn} h_n$ and $\|h_n\|_{L^2(\gamma_1)}=1$,
$$
\|h_n\|_{L^q(\gamma_1)}
= e^{sn} \,\| e^{s L_{\rm OU}} h_n \|_{L^q(\gamma_1)}
\le e^{sn} =(q-1)^{n/2}.
$$

We now turn to the exponential-rate statements. For $2\le p<q$, applying \eqref{eq:upper-lower-hermite} with $q$ replaced by $p$ (and noting that $\|h_n\|_{L^2(\gamma_1)}=1$) yields \eqref{eq:ratio-hermite-I}. If $1<p<2$, then $\|h_n\|_{L^p(\gamma_1)}\le \|h_n\|_{L^2(\gamma_1)}=1$, so
\[
\liminf_{n \to \infty} \frac{\log\big(\| h_n \|_{L^q(\gamma_1)}/ \| h_n \|_{L^p(\gamma_1)}  \big)}{n} \ge \lim_{n \to \infty} \frac{\log\big(\| h_n \|_{L^q(\gamma_1)}  \big)}{n} = \frac12\log(q-1),
\]
where the last identity follows from \eqref{eq:upper-lower-hermite}. This completes the proof.
\end{proof}

\begin{remark}
By \cite[Theorem~2.1]{L02}, for any $1 \le q < \infty$ with $q \ne 2$ one has, as $n \to \infty$,
$$
\| h_n \|_{L^q(\gamma_1)} = \frac{c_q}{n^{1/4}} \max\{ (q-1)^{n/2} \,, 1\}  \left( 1 + O\!\left( \frac{1}{n} \right) \right).
$$
This shows that the lower bound in \eqref{eq:upper-lower-hermite} is sharp up to constants; moreover, \eqref{eq:ratio-hermite-II} holds with $\liminf$ replaced by $\lim$, and equals $\frac12\log(q-1)$.
\end{remark}

We can now combine Corollary~\ref{cor:SW-spectral-obstruction-II} with Proposition~\ref{prop:Hermite-exp-growth} to conclude that, for any $1 < p < q <\infty$ and $t>0$, the Poisson--Hermite semigroup is not $L^p \to L^q$ bounded. We first treat the case $q > 2$. In dimension one, \eqref{eq:ratio-hermite-I} and \eqref{eq:ratio-hermite-II} immediately imply that
\[
\kappa_{p, q} = \limsup_{n \to \infty} \frac{\log \big( \| h_n \|_{L^q(\gamma_1)}/ \| h_n \|_{L^p(\gamma_1)} \big)}{\sqrt{n}} = +\infty.
\]
Hence Corollary~\ref{cor:SW-spectral-obstruction-II} yields $\| P_t \|_{p \to q} = +\infty$. The $d$-dimensional case follows by tensorization. Indeed, for $n \ge 1$ define $F_n(x) := h_n(x_1)$ on $\mathbb R^d$. Then $\| F_n \|_{L^r(\gamma_d)} = \| h_n \|_{L^r(\gamma_1)}$ for $r \in \{p, q\}$, and $L_{\rm OU} F_n = -n F_n$. Consequently, the higher-dimensional case follows from Lemma~\ref{lem:unified-eigen-test} and the one-dimensional case. It remains to consider the case $q\le2$. Since $P_t$ is self-adjoint on $L^2(\gamma_d)$, duality gives
\[
\|P_t\|_{p\to q}=\|P_t\|_{q'\to p'},
\]
where $p' = \frac{p}{p-1}$ and $q' = \frac{q}{q-1}$. If $p< q \le 2$, then $p<2$ and hence $p'>2$. Applying the previous argument to  $(q',p')$,
we obtain $\|P_t\|_{q'\to p'}=+\infty$, and the proof is complete.

We now turn to the Poisson--Laguerre semigroup. In analogy with Proposition~\ref{prop:Hermite-exp-growth}, we first prove the following $L^q$-estimate for the Laguerre polynomials $(\ell_n^\alpha)_{n \ge 0}$.

\begin{proposition}\label{prop:Laguerre-exp-growth}
Fix $q > 2$ and $\alpha > -1$. For every $\rho \in (0, q -1)$ there exists a constant $c_{\alpha, q, \rho} > 0$ such that for every $n \ge 1$,
\begin{equation} \label{eq:lower-laguerre}
\| \ell_n^\alpha \|_{L^q(\mu_\alpha)} \geq c_{\alpha, q, \rho}\, \rho^n n^{\alpha/2}.
\end{equation}
Consequently, for $1 < p \le 2$,
\begin{equation} \label{eq:ratio-laguerre-I}
\liminf_{n \to \infty} \frac{\log \big( \| \ell_n^\alpha \|_{L^q(\mu_\alpha)}/ \| \ell_n^\alpha \|_{L^p(\mu_\alpha)} \big)}{n} \ge \log(q-1).
\end{equation}
If moreover $\alpha \ge -1/2$, then for every $n \ge 1$,
\begin{equation} \label{eq:upper-laguerre}
\| \ell_n^\alpha \|_{L^q(\mu_\alpha)} \le (q-1)^n.
\end{equation}
Consequently, if $\alpha \ge -1/2$, then for $2 \le p < q$,
\begin{equation} \label{eq:ratio-laguerre-II}
\lim_{n \to \infty} \frac{\log \big( \| \ell_n^\alpha \|_{L^q(\mu_\alpha)}/ \| \ell_n^\alpha \|_{L^p(\mu_\alpha)} \big)}{n}
= \log\bigg( \frac{q-1}{p-1} \bigg).
\end{equation}
\end{proposition}

\begin{proof}
We first prove \eqref{eq:lower-laguerre} and \eqref{eq:upper-laguerre}. Write $q'=\frac{q}{q-1}$. By duality,
\begin{equation} \label{eq:dual-laguerre}
\| \ell_n^\alpha \|_{L^q(\mu_\alpha)} = \sup_{\|g\|_{L^{q'}(\mu_\alpha)}=1} \big| \langle \ell_n^\alpha, g \rangle_{\mu_\alpha} \big|.
\end{equation}
For $b\in(0,1/q')$, we define
$$
g_b(x):=(1-q'b)^{(\alpha+1)/q'}\,e^{bx},\qquad x>0.
$$
Since $\int_0^\infty e^{sx}\, d \mu_\alpha(x) =(1-s)^{-(\alpha+1)}$ for $s<1$, we have $\|g_b\|_{L^{q'}(\mu_\alpha)}=1$.

Integrating the generating function \eqref{eq:Laguerre-generating} against $e^{bx}\, d \mu_\alpha$ and comparing coefficients gives
\begin{equation} \label{eq:fourier-expansion-Laguerre}
\int_0^\infty L_n^\alpha(x)\,e^{bx}\, d \mu_\alpha(x)
=
(-1)^n\,\frac{\Gamma(n+\alpha+1)}{\Gamma(\alpha+1)\,n!}\,
\frac{b^n}{(1-b)^{n+\alpha+1}}.
\end{equation}
Recall that $\ell_n^\alpha=\Big(\frac{n!\Gamma(\alpha+1)}{\Gamma(n+\alpha+1)}\Big)^{1/2}L_n^\alpha$. Therefore, by \eqref{eq:dual-laguerre},
\begin{align*}
\|\ell_n^\alpha\|_{L^q(\mu_\alpha)} &\ge
\big| \langle \ell_n^\alpha, g_b \rangle_{\mu_\alpha} \big|\\
& = (1-q'b)^{(\alpha+1)/q'}(1-b)^{-(\alpha+1)}\,
\Big(\frac{\Gamma(n+\alpha+1)}{\Gamma(\alpha+1)\,n!}\Big)^{1/2}
\Big(\frac{b}{1-b}\Big)^n.
\end{align*}
Choose $b = \frac{\rho}{1 + \rho} \in (0, 1/q')$. By Stirling's formula, $\Gamma(n+\alpha+1)/n!\simeq n^\alpha$, we thus obtain the stated lower bound \eqref{eq:lower-laguerre}.

If $\alpha \ge -1/2$, set $s : = \log(q - 1)> 0$. By \eqref{eq:pre-hypercontractivity} with $p=2$ we have $\big\| e^{s L_\alpha} \big\|_{2 \to q}\le 1$. Since $e^{s L_\alpha} \ell_n^\alpha = e^{-sn} \ell_n^\alpha$ and $\| \ell_n^\alpha \|_{L^2(\mu_\alpha)}=1$,
$$
\| \ell_n^\alpha \|_{L^q(\mu_\alpha)}
= e^{sn} \,\| e^{s L_\alpha} \ell_n^\alpha \|_{L^q(\mu_\alpha)}
\le e^{sn} = (q-1)^{n}.
$$

We now turn to the exponential-rate statements \eqref{eq:ratio-laguerre-I} and \eqref{eq:ratio-laguerre-II}. The former follows immediately from \eqref{eq:lower-laguerre} and the bound $\|\ell_n^\alpha\|_{L^p(\mu_\alpha)}\le \|\ell_n^\alpha\|_{L^2(\mu_\alpha)}=1$. If $\alpha\ge -1/2$, then \eqref{eq:ratio-laguerre-II} follows by combining \eqref{eq:lower-laguerre} and \eqref{eq:upper-laguerre} for both $p$ and $q$ (and using $\|\ell_n^\alpha\|_{L^2(\mu_\alpha)}=1$). This completes the proof.
\end{proof}

We now prove the statement of Theorem~\ref{thm:intro-main1} concerning the Poisson--Laguerre semigroup. We first consider the case $\alpha \ge -1/2$. As in the Poisson--Hermite case, by duality it suffices to treat $q > 2$. It follows from Proposition \ref{prop:Laguerre-exp-growth} that, for any $1<p<q<\infty$ with $q>2$,
\begin{equation} \label{eq:blow-up-laguerre}
\kappa_{p, q} = \limsup_{n \to \infty} \frac{\log \big( \| \ell_n^\alpha \|_{L^q(\mu_\alpha)}/ \| \ell_n^\alpha \|_{L^p(\mu_\alpha)} \big)}{\sqrt{n}} = +\infty
\end{equation}
Then Corollary~\ref{cor:SW-spectral-obstruction-II} implies that $\| P^\alpha_t \|_{p \to q} = +\infty$ for any $t > 0$. The case $\alpha > -1$ can be viewed as a particular case of Theorem~\ref{thm:intro-main2}; we refer Section~\ref{sec:blow-up-ou-lauguerre} for the proof.

\subsection{Jacobi: ultracontractivity} \label{sec:jacobi}

In contrast to the hypercontractivity in the Nelson--Gross sense, the following notion of ultracontractivity is preserved under Poisson subordination.

\begin{theorem}\label{thm:ultra-poisson}
Let $(T_s)_{s\ge0}$ be a semigroup of bounded linear operators on $L^2(\mu)$ over a probability space $(X,\mu)$. Assume that there exists $\sigma > 0$ such that for any $s > 0$,
\begin{align} \label{eq:ultra}
\| T_s \|_{1 \to \infty} \lesssim (1 \wedge s)^{-\sigma/2}.
\end{align}
The Poisson subordination of $(T_s)_{s\ge0}$ is given by
$$
P_t = \frac{1}{2\sqrt{\pi}} \int_{0}^{\infty} \frac{t}{s^{3/2}} \, e^{ - \frac{t^2}{4s}} \, T_s \, d s, \quad t > 0,
$$
Then, for any $t > 0$,
$$
\| P_t \|_{1 \to \infty} \lesssim   (1 \wedge t)^{- \sigma}.
$$
\end{theorem}

\begin{proof}
By the definition of $P_t$ and \eqref{eq:ultra}, we have
\begin{align*}
\| P_t \|_{1 \to \infty} \le \frac{1}{2\sqrt{\pi}} \int_{0}^{\infty} \frac{t}{s^{3/2}} \, e^{ - \frac{t^2}{4s}} \, \| T_s \|_{1 \to \infty} \, d s \lesssim \int_{0}^{1} \frac{t}{s^{(\sigma + 3)/2}} \, e^{ - \frac{t^2}{4s}} \, d s \, + \, \int_{1}^{\infty} \frac{t}{s^{3/2}} \, e^{ - \frac{t^2}{4s}} \, \, d s.
\end{align*}
Set $u = \frac{t^2}{4s}$, then
\begin{align*}
\int_{0}^{1} \frac{t}{s^{(\sigma + 3)/2}} \, e^{ - \frac{t^2}{4s}} \, d s = 4^{\frac{\sigma+1}{2}} \cdot t^{-\sigma} \int_{\frac{t^2}{4}}^{\infty} u^{\frac{\sigma - 1}{2}} e^{-u} d u =  \begin {cases}
O(t^{-\sigma}) & \text {if \, $t \in (0, 1)$,}\\
O(1)      & \text{if \, $t \in [1, \infty)$.}
\end {cases}
\end{align*}
Combining this with
\begin{align*}
\int_{1}^{\infty} \frac{t}{s^{3/2}} \, e^{ - \frac{t^2}{4s}} \, d s \lesssim 1,
\end{align*}
we complete the proof.
\end{proof}

Applying this theorem to the Poisson--Jacobi semigroup with $\alpha,\beta \ge -1/2$, we obtain from \eqref{eq:jacobi-ultra} that \eqref{eq:ultra} holds with $\sigma = 2(\max\{\alpha,\beta\}+1)$. Hence, Theorem~\ref{thm:ultra-poisson} yields that for any $t > 0$,
\[
\big\| P_t^{\alpha, \beta} \big\|_{1 \to \infty} \lesssim (1 \wedge t)^{-2(\max\{ \alpha, \beta \} +1 )}.
\]
Since $P_t^{\alpha, \beta}$ is an $L^r$-contraction for all $r \in [1, \infty]$, the proof follows by interpolation.

\subsection{Degeneration of hypercontractive bounds in Jacobi limit transitions}\label{sec:unified-characterization}

The three classical polynomial families are linked by standard limit relations (see e.g. \cite{B96, BGL}). We refer to \cite[Appendix~B, (B.46) and (B.49)]{UrbWorkshop} for the two transitions used here, namely Jacobi$\to$Laguerre and ultraspherical/Gegenbauer$\to$Hermite. In this section we discuss the degeneration of hypercontractive bounds in these limits. We begin with the following limit relations for the corresponding eigenpolynomials, using the notation introduced in Section~\ref{sec:facts}.

\begin{proposition}\label{prop:askey-limits}
Fix $n\in\mathbb N_0$.
\begin{enumerate}
\item[(i)] Fix $\alpha>-1$. Then for every $x\ge 0$,
$$
\lim_{\beta\to\infty} J_n^{(\alpha,\beta)}\!\Bigl(1-\frac{2x}{\beta}\Bigr)=L_n^\alpha(x).
$$

\item[(ii)] Let $\lambda>0$.  We write $C_n^\lambda$ for the Gegenbauer (ultraspherical) polynomial normalized as in
\cite[Appendix~B, (B.47)]{UrbWorkshop}, i.e.
$$
C_n^\lambda(x)
:= \frac{\Gamma(\lambda+\tfrac12)\,\Gamma(n+2\lambda)}{\Gamma(2\lambda)\,\Gamma(n+\lambda+\tfrac12)}\,
J_n^{(\lambda-\frac12,\lambda-\frac12)}(x),
\qquad x\in[-1,1].
$$
Then for every $x\in\mathbb R$,
$$
\lim_{\lambda\to\infty}\lambda^{-n/2}\,C_n^\lambda\!\Bigl(\frac{x}{\sqrt{\lambda}}\Bigr)
=\frac{1}{n!}\,H_n(x).
$$
\end{enumerate}
\end{proposition}

\begin{remark}\label{rem:askey-orthonormal-note}
For convenience, we work with their $L^2$-normalized versions below.  One can also phrase
the limits at the level of orthonormal polynomials; see \cite[Appendix~B]{UrbWorkshop}.
\end{remark}

Next, we introduce two rescaled Jacobi probability spaces that implement these limit relations at the level of $L^p$--norms.
The corresponding time rescalings  $t \mapsto t/\sqrt{\beta}$ and $t \mapsto t/\sqrt{2\lambda}$ are chosen so that
$\sqrt{\lambda_n^{(\alpha,\beta)}}$ remains of order $\sqrt{n}$ along the limit, hence the Poisson multipliers do not
collapse trivially.

\medskip\noindent\textbf{Jacobi$\to$Laguerre rescaling near $x=1$.}
Fix $\alpha>-1$ and let $\beta>-1$ vary.  Recall the Jacobi probability measure
$$
d \mu_{\alpha,\beta}(x)
= Z_{\alpha,\beta}^{-1} \, (1-x)^\alpha(1+x)^\beta\mathbf 1_{(-1,1)}(x)\, dx,
\qquad
Z_{\alpha,\beta}=2^{\alpha+\beta+1}B(\alpha+1,\beta+1).
$$
Define the change of variables
$$
S_\beta:(0,\beta)\to(-1,1),
\qquad
S_\beta(y):=1-\frac{2y}{\beta},
\qquad
S_\beta^{-1}(x)=\frac{\beta}{2}(1-x).
$$
Let $\widehat\mu_{\alpha,\beta}$ be the push-forward probability measure on $(0,\beta)$: for any Borel set $E \subset (0,\beta)$,
$$
\widehat\mu_{\alpha,\beta}(E):=\mu_{\alpha,\beta}(S_\beta(E)).
$$
Then $\widehat\mu_{\alpha,\beta}$ admits the density
$$
d \widehat\mu_{\alpha,\beta}(y)
=
\widehat c_{\alpha,\beta}\,y^\alpha\Bigl(1-\frac{y}{\beta}\Bigr)^\beta
\mathbf 1_{(0,\beta)}(y)\,dy,
\qquad
\widehat c_{\alpha,\beta}
=
\frac{1}{B(\alpha +1, \beta + 1)}\,\beta^{-(\alpha+1)}.
$$
The associated pullback map
$$
U_\beta:L^p(\widehat\mu_{\alpha,\beta})\to L^p(\mu_{\alpha,\beta}),
\qquad
(U_\beta f)(x):=f\!\Bigl(\frac{\beta}{2}(1-x)\Bigr),
$$
is an isometry for each $p\in[1,\infty]$:
$$
\|U_\beta f\|_{L^p(\mu_{\alpha,\beta})} = \|f\|_{L^p(\widehat\mu_{\alpha,\beta})}.
$$

\medskip\noindent\textbf{Ultraspherical/Gegenbauer$\to$Hermite rescaling near $x=0$.}
Let $\lambda>0$ and consider the symmetric Jacobi parameters
$$
\alpha=\beta=\lambda-\frac12,
\qquad
d \mu_{\lambda-\frac12,\lambda-\frac12}(x)
=Z_{\lambda-\frac12,\lambda-\frac12}^{-1} \, (1-x^2)^{\lambda-\frac12}\mathbf 1_{(-1,1)}(x)\, dx,
$$
which correspond to the ultraspherical/Gegenbauer regime (cf.\ \cite[Appendix~B, (B.47)]{UrbWorkshop}).
Define the scaling
$$
R_\lambda:(-\sqrt{\lambda},\sqrt{\lambda})\to(-1,1),
\qquad
R_\lambda(y):=\frac{y}{\sqrt{\lambda}},
\qquad
R_\lambda^{-1}(x)=\sqrt{\lambda}\,x.
$$
Let $\widehat\gamma_\lambda$ be the push-forward probability measure on $(-\sqrt{\lambda},\sqrt{\lambda})$: for any Borel set $E \subset (-\sqrt{\lambda},\sqrt{\lambda})$,
$$
\widehat\gamma_\lambda(E):= \mu_{\lambda-\frac12,\lambda-\frac12}(R_\lambda(E)).
$$
Then $\widehat\gamma_\lambda$ admits the density
$$
d \widehat\gamma_\lambda(y)
=
\widehat d_\lambda
\Bigl(1-\frac{y^2}{\lambda}\Bigr)^{\lambda-\frac12}
\mathbf 1_{(-\sqrt{\lambda},\sqrt{\lambda})}(y)\, dy,
\qquad
\widehat d_\lambda:=\frac{1}{\sqrt{\lambda}\, Z_{\lambda-\frac12,\lambda-\frac12}}.
$$
The pullback map
$$
V_\lambda:L^p(\widehat\gamma_\lambda)\to L^p(\mu_{\lambda-\frac12,\lambda-\frac12}),
\qquad
(V_\lambda f)(x):=f(\sqrt{\lambda}\,x),
$$
is an $L^p$-isometry for all $p\in[1,\infty]$:
$$
\|V_\lambda f\|_{L^p(\mu_{\lambda-\frac12,\lambda-\frac12})}=\|f\|_{L^p(\widehat\gamma_\lambda)}.
$$

The main theorem of this section is as follows.

\begin{theorem}\label{thm:unified-characterization}
Fix $\alpha \ge -1/2$, $1 < p < q < \infty$ and $t>0$. Consider the two rescaled Poisson--Jacobi semigroups
\begin{align}
\widetilde P_t^{(\alpha,\beta)}
&:=U_\beta^{-1}\,P_{t/\sqrt{\beta}}^{\alpha,\beta}\,U_\beta
\quad\text{acting on }L^2(\widehat\mu_{\alpha,\beta}),\\
\widetilde P_t^{(\lambda)}
&:=V_\lambda^{-1}\,P_{t/\sqrt{2\lambda}}^{\lambda-\frac12,\lambda-\frac12}\,V_\lambda
\quad\text{acting on }L^2(\widehat\gamma_\lambda),
\end{align}
where $U_\beta,V_\lambda$ are the $L^p$-isometries defined above.
Then
$$
\big\|\widetilde P_t^{(\alpha,\beta)} \big\|_{L^p(\widehat\mu_{\alpha,\beta})\to L^q(\widehat\mu_{\alpha,\beta})}\xrightarrow[\beta\to\infty]{}+\infty,
\qquad
\big\| \widetilde P_t^{(\lambda)} \big\|_{L^p(\widehat\gamma_{\lambda})\to L^q(\widehat\gamma_{\lambda})}\xrightarrow[\lambda\to\infty]{}+\infty.
$$
\end{theorem}

\begin{proof}
We treat the Jacobi$\to$Laguerre limit; the Gegenbauer$\to$Hermite limit is analogous. Since $\widetilde P_t^{(\alpha,\beta)}$ is self-adjoint on $L^2(\widehat\mu_{\alpha,\beta})$, by duality we only need to consider the regime $q > 2$.

Let $\lambda_n^{(\alpha,\beta)}=n(n+\alpha+\beta+1)$ be the Jacobi eigenvalues.  Then for each fixed $n$,
$$
\lim_{\beta\to\infty}  \frac{1}{\sqrt{\beta}}\sqrt{\lambda_n^{(\alpha,\beta)}} = \sqrt{n}.
$$
Let $\widehat {\jj}_{n,\beta} := U_\beta^{-1} \jj_n^{(\alpha,\beta)}$, so
$\|\widehat {\jj}_{n,\beta}\|_{L^2(\widehat\mu_{\alpha,\beta})} = 1$ and
$$
\widetilde P_t^{(\alpha,\beta)} \, \widehat \jj_{n,\beta}
=
e^{-(t/\sqrt{\beta})\sqrt{\lambda_n^{(\alpha,\beta)}}}\,\widehat \jj_{n,\beta}.
$$
Consequently, for each $n \in \mathbb N_0$,
$$
\|\widetilde P_t^{(\alpha,\beta)}\|_{L^p(\widehat\mu_{\alpha,\beta})\to L^q(\widehat\mu_{\alpha,\beta})}
\ge
e^{-(t/\sqrt{\beta})\sqrt{\lambda_n^{(\alpha,\beta)}}}\,
\|\widehat \jj_{n,\beta}\|_{L^q(\widehat\mu_{\alpha,\beta})}/\|\widehat \jj_{n,\beta}\|_{L^p(\widehat\mu_{\alpha,\beta})}.
$$

By the Jacobi-to-Laguerre limit relation in Proposition~\ref{prop:askey-limits}\,\textup{(i)} and dominated convergence on the rescaled probability spaces,
for each fixed $n$ and each finite $r\ge1$ one has
$\widehat \jj_{n,\beta}\to \ell_n^\alpha$ in $L^r(\mu_\alpha)$, hence
$\|\widehat \jj_{n,\beta}\|_{L^r(\widehat\mu_{\alpha,\beta})} \to \|\ell_n^\alpha\|_{L^r(\mu_\alpha)}$ for $r \in \{ p, q \}$.
Therefore, for each fixed $n$, as $\beta \to \infty$,
$$
e^{-(t/\sqrt{\beta})\sqrt{\lambda_n^{(\alpha,\beta)}}}\,
\|\widehat \jj_{n,\beta}\|_{L^q(\widehat\mu_{\alpha,\beta})}/\|\widehat \jj_{n,\beta}\|_{L^p(\widehat\mu_{\alpha,\beta})}
\longrightarrow
e^{-t\sqrt{n}} \,\|\ell_n^\alpha\|_{L^q(\mu_\alpha)} / \|\ell_n^\alpha\|_{L^p(\mu_\alpha)}.
$$
By \eqref{eq:blow-up-laguerre}, for any $1 <p < q <\infty$ with $q > 2$,
$$\sup_{n\ge1}e^{-t\sqrt{n}} \,\|\ell_n^\alpha\|_{L^q(\mu_\alpha)} / \|\ell_n^\alpha\|_{L^p(\mu_\alpha)}  =+\infty.$$
Given $M>0$, choose $n$ so that $e^{-t\sqrt{n}} \,\|\ell_n^\alpha\|_{L^q(\mu_\alpha)} / \|\ell_n^\alpha\|_{L^p(\mu_\alpha)}>4M$.
For this $n$, the convergence above yields $\beta_0$ such that for all $\beta\ge\beta_0$,
$$
e^{-(t/\sqrt{\beta})\sqrt{\lambda_n^{(\alpha,\beta)}}}\,
\|\widehat \jj_{n,\beta}\|_{L^q(\widehat\mu_{\alpha,\beta})}/\|\widehat \jj_{n,\beta}\|_{L^p(\widehat\mu_{\alpha,\beta})}>2M.
$$
Hence $\|\widetilde P_t^{(\alpha,\beta)}\|_{p\to q}>2M$ for all $\beta\ge\beta_0$, and the first limit follows.

The ultraspherical/Gegenbauer$\to$Hermite statement is proved in the same way, using Proposition~\ref{prop:askey-limits}\,\textup{(ii)}
and the corresponding rescaled probability measures.
\end{proof}

\begin{remark}\label{rem:unified-intuition}
The degeneration statement is consistent with the limit relations in Proposition~\ref{prop:askey-limits}, once the time parameter is rescaled
so that $\sqrt{\lambda_n^{(\alpha, \beta)}}$ remains of order $\sqrt{n}$.
\end{remark}

\section{Proof of Theorem~\ref{thm:intro-main2}} \label{sec:OU-bernstein-drift}

Before turning to the proof of Theorem~\ref{thm:intro-main2}, we need a factorization lemma for the subordinated semigroup. Let $(T_s)_{s\ge 0}$ be a symmetric Markov semigroup with generator $L$ and invariant probability measure $\mu$. Assume that $L$ is self-adjoint on $L^2(\mu)$. Let $f$ be a Bernstein function with L\'evy--Khintchine triplet $(a, b, \nu)$, and let $S_t^{\, f} : = e^{-t f(-L)}$ denote the subordinated semigroup. Then the following lemma holds.

\begin{lemma}\label{lem:OU-drift-factorization}
Define the nonlinear part of $f$ by
$$
f_0(\lambda):=f(\lambda)-a-b\lambda = \int_{(0,\infty)} \bigl(1- e^{-s\lambda}\bigr) \,d \nu(s).
$$
Then $f_0$ is a Bernstein function with $f_0(0)=0$ and $f_0(\lambda)=o(\lambda)$ as
$\lambda\to\infty$. Moreover, for any $t\ge0$,
\begin{equation} \label{eq:OU-factorization}
S_t^{\,f} =e^{-at}\, e^{-t f_0(-L)}\, e^{tb L} = e^{-at}\, e^{tb L}\, e^{-t f_0(-L)}.
\end{equation}
Finally, $(e^{-t f_0(-L)})_{t\ge0}$ is a symmetric Markov semigroup and hence
$$
\big\| e^{-t f_0(-L)} \big\|_{r \to r} \le 1
\qquad\text{for all }r\in[1,\infty].
$$
\end{lemma}

\begin{proof}
By the L\'evy--Khintchine formula \eqref{eq:bernstein-LK}, $f_0$ is a Bernstein function with L\'evy--Khintchine triplet $(0,0,\nu)$; in particular $f_0(0)=0$.  By Bochner subordination,
$$
e^{-t f_0(-L)}= \int_0^\infty T_s \,d \eta_t(s)
$$
for a probability measure $\eta_t$. Since this is an average of contractions, it follows that $\big\|e^{-t f_0(-L)} \big\|_{r\to r} \le 1$ for all $r\in[1,\infty]$.

Now we prove $f_0(\lambda)=o(\lambda)$ as $\lambda\to\infty$. Using
$1-e^{-s\lambda}\le \min\{1,s\lambda\}$ yields
$$
0\le \frac{f_0(\lambda)}{\lambda}
=
\int_{(0,\infty)}\frac{1-e^{-s\lambda}}{\lambda}\, d \nu(s)
\le
\int_{(0,\infty)}(\lambda^{-1} \wedge s)\, d \nu(s) \xrightarrow[\lambda\to\infty]{}0,
$$
where we use dominated convergence and $\int_{(0,\infty)}(1\wedge s)\,d \nu(s) <\infty$.

The factorization \eqref{eq:OU-factorization} follows from functional calculus and
commutativity of Borel functions of the same self-adjoint operator $L$. The proof is complete.
\end{proof}

\begin{remark}
As a direct consequence of Lemma~\ref{lem:OU-drift-factorization}, for a Bernstein function $f$ with L\'evy--Khintchine triplet $(a, b, \nu)$ we have
\[
b = \lim_{\lambda \to \infty} \frac{f(\lambda)}{\lambda}.
\]
\end{remark}

\subsection{Ornstein--Uhlenbeck and Laguerre cases} \label{sec:blow-up-ou-lauguerre}

We begin with the Ornstein--Uhlenbeck case. Let $\gamma_d$ be the Gaussian measure on $\mathbb R^d$ and $L_{\rm OU}$ the Ornstein--Uhlenbeck operator, as in Section \ref{sec:facts}. We first treat the special case $b=0$, equivalently $f(\lambda)=o(\lambda)$ as $\lambda\to\infty$. Our goal is to prove that $\big\| S_t^{\,f} \big\|_{p \to q} = +\infty$ for any $1 < p < q < \infty$ and $t >0$. By duality and the self-adjointness of $S_t^{\,f}$ on $L^2(\gamma_d)$,
\begin{equation} \label{eq:duality}
\big\| S_t^{\,f} \big\|_{p \to q} = \big\| S_t^{\,f} \big\|_{q' \to p'}.
\end{equation}
Therefore, it suffices to restrict ourselves to the regime $q> 2$. By the tensorization argument of Section~\ref{sec:blow-up}, we may reduce the proof to the one-dimensional case. We then apply Theorem~\ref{thm:SW-spectral-obstruction}(i). Indeed, it follows from Proposition~\ref{prop:Hermite-exp-growth} and the assumption
$f(n) = o(n)$ as $n \to \infty$ that, for any $1 < p < q<\infty$ with $q > 2$,
\begin{align*}
\Theta_{p, q, f} &= \limsup_{n \to \infty} \frac{\log \big(\| h_n \|_{L^q(\gamma_1)}/ \| h_n \|_{L^p(\gamma_1)} \big)}{f(n)}\\
 & = \limsup_{n \to \infty} \frac{\log\big(\| h_n \|_{L^q(\gamma_1)}/ \| h_n \|_{L^p(\gamma_1)}  \big)}{n} \cdot \frac{n}{f(n)} = +\infty,
\end{align*}
which, by Theorem~\ref{thm:SW-spectral-obstruction}(i), implies that $\big\|S_t^{\,f} \big\|_{p\to q}=+\infty$ for any $t>0$.

We turn to the case $b > 0$. With $t>0$ fixed, we split the proof into two cases.

\noindent\textbf{Case 1: $q \le 1 + (p-1) e^{2bt}$.} By Lemma~\ref{lem:OU-drift-factorization} applied to $L_{\rm OU}$, we have the factorization
$$
S_t^{\,f} = e^{-at}\, e^{-t f_0(-L_{\rm OU})}\, e^{tb  L_{\rm OU}}.
$$
Since $e^{tb  L_{\rm OU}}$ is the Ornstein-Uhlenbeck semigroup at time $bt$, the hypercontractive estimate \eqref{eq:intro-hypercontractivity} yields $\big\| e^{tb  L_{\rm OU}} \big\|_{p\to q} \leq 1$ whenever $q \le 1 + (p-1)e^{2bt}$. Moreover, $e^{-t f_0(-L_{\rm OU})}$ is an $L^q$-contraction (see Lemma~\ref{lem:OU-drift-factorization}), hence for any $q \le 1 + (p-1)e^{2bt}$,
$$
\big\| S_t^{\,f} \big\|_{p \to q} \le e^{-at} \, \big\| e^{-t f_0(-L_{\rm OU})}  \big\|_{q \to q} \, \big\| e^{tb  L_{\rm OU}} \big\|_{p\to q} \le e^{-at}.
$$

On the other hand, $S_t^{\, f} \mathbf 1 = e^{-at} \mathbf 1$, so
\[
\big\| S_t^{\,f} \big\|_{p \to q} \ge \frac{\big\| S_t^{\, f} \mathbf 1 \big\|_{L^q(\gamma_d)}}{\| \mathbf 1 \|_{L^p(\gamma_d)}} = e^{-at}.
\]
Therefore $\big\| S_t^{\,f} \big\|_{p \to q}=e^{-at}$ for all $q\le 1 + (p-1)e^{2bt}$.

\noindent\textbf{Case 2: $q > 1 + (p-1) e^{2bt}$.} By tensorization as before, it suffices to treat the one-dimensional case. The method combining Theorem~\ref{thm:SW-spectral-obstruction}(ii) with Proposition~\ref{prop:Hermite-exp-growth} does not yield a satisfactory obstruction when $1 <p< 2$. To overcome this difficulty, we adopt a different testing argument. Namely, we consider the exponential family $g_\tau(x) := \exp(\tau x)$ with $\tau > 0$ and test via the bilinear identity: if $T : L^p(\mu)\to L^q(\mu)$ is a bounded linear operator for some $1<p,q<\infty$, then
\begin{equation} \label{eq:bilinear}
\|  T \|_{p \to q} = \sup_{g \in L^p(\mu), \, h \in L^{q'}(\mu)} \frac{\langle T g ,h \rangle_{\mu}}{\| g \|_{L^p(\mu)} \, \| h \|_{L^{q'}(\mu)}}.
\end{equation}

For $\tau > 0$, a direct computation yields that for any $1 \leq r < \infty$,
\begin{equation} \label{eq:expontial-lr-norm}
\| g_\tau \|_{L^r(\gamma_1)} = \exp \!\Big(\frac{r\tau^2}{4}\Big).
\end{equation}
By \eqref{eq:fourier-expansion-OU}, $g_\tau$ admits the Hermite expansion
$$
g_\tau(x) = \exp \!\Big( \frac{\tau^2}{4} \Big) \bigg( 1 + \sum_{n = 1}^{\infty}\frac{\tau^n}{\sqrt{2^n n!}} h_n(x) \bigg).
$$
Since $(h_n)_{n\ge 0}$ forms an orthonormal basis of $L^2(\gamma_1)$ and $S_t^{\, f} h_n = e^{-tf(n)} h_n$, Parseval's relation yields that for any $\tau_1, \tau_2 > 0$,
\begin{equation} \label{eq:basis}
\langle S_t^{\, f} g_{\tau_1} , g_{\tau_2} \rangle_{\gamma_1} = \exp\! \Big( \frac{\tau_1^2 + \tau_2^2}{4} \Big) \bigg( e^{-t f(0)} + \sum_{n = 1}^{\infty}\frac{(\tau_1 \tau_2)^n}{2^n n!} e^{-t f(n)} \bigg) = \exp\! \Big( \frac{\tau_1^2 + \tau_2^2}{4} \Big) F_t\Big( \frac{\tau_1 \tau_2}{2}  \Big),
\end{equation}
where
\[
F_t(z) := e^{-t f(0)} + \sum_{n = 1}^{\infty}\frac{e^{-t f(n)}}{n!} z^n.
\]

Recall that $f(\lambda) = b \lambda + o(\lambda)$ as $\lambda \to \infty$. We claim that, for each fixed $t> 0$,
\begin{equation} \label{eq:asy-of-F}
\lim_{z \to +\infty} \frac{\log F_t (z)}{z} = e^{-b t}.
\end{equation}
By \eqref{eq:bernstein-LK}, we have $f(n) \ge  b n$ for each $n \in \mathbb N$. Hence, for $z \ge 0$,
\[
F_t(z) \le 1 + \sum_{n = 1}^{\infty}\frac{e^{-t b n}}{n!} z^n  = \exp\! \Big(z e^{-b t} \Big),
\]
and therefore
\[
\limsup_{z \to +\infty} \frac{\log F_t (z)}{z} \le e^{-b t}.
\]
On the other hand, fix $\varepsilon > 0$, by $f(n) = bn + o(n)$, there exists $M > 0$ such that for each $n \ge M$, $f(n) < (b + \varepsilon ) n$. For $n \ge M$, we set $z_n :=  e^{(b + \varepsilon)t} n$. Then for any $z_n \le z < z_{n+1}$,
\[
F_t(z) \ge F_t(z_n) \ge \frac{z_n^n}{n!} \, e^{-t f(n)} > \frac{(e^{(b + \varepsilon)t} n)^n}{n!} e^{-t (b + \varepsilon) n} = \frac{n^n}{n!}.
\]
By Stirling's formula, $\log (n^n/n!) = n + O(\log n)$. Since $n =  e^{-(b + \varepsilon)t} z + O(1)$, this yields
\[
\log F_t(z) \ge z e^{- (b + \varepsilon)t} + O(\log z),
\]
and consequently,
\[
\liminf_{z \to +\infty} \frac{\log F_t (z)}{z} \ge e^{-  (b + \varepsilon)t}.
\]
Letting $\varepsilon \downarrow 0$ and combining with the limsup bound, we obtain \eqref{eq:asy-of-F}, which implies that, as $z \to +\infty$,
\[
\log F_t (z) = z e^{-bt} + o(z).
\]

Now, by \eqref{eq:expontial-lr-norm} and \eqref{eq:basis}, as $\tau_1 \tau_2 \to +\infty$,
\begin{equation}
\begin{split}
&\log \left( \frac{\langle S_t^{\, f} g_{\tau_1} , g_{\tau_2} \rangle_{\gamma_1}}{\| g_{\tau_1} \|_{L^p(\gamma_1)} \, \| g_{\tau_2} \|_{L^{q'}(\gamma_1)}}  \right) = -\frac14 \bigg( (p-1) \tau_1^2 + (q' - 1) \tau_2^2  \bigg) + \log F_t  \Big( \frac{\tau_1 \tau_2}{2}  \Big)
\\
& \quad = -\frac14 \bigg( (p-1) \tau_1^2 + (q' - 1) \tau_2^2 - 2 e^{-bt} \tau_1 \tau_2 \bigg) + o(\tau_1 \tau_2).
\end{split}
\end{equation}
If $q > 1 + (p -1) e^{2bt}$, namely $e^{-2bt} - (p-1)(q'-1) > 0$, there exists $k > 0$ such that
\[
(p-1) k^2  - 2 e^{-bt} k + (q' - 1) < 0.
\]
We take $\tau_1 = k \tau_2$ and let $\tau_2 \to +\infty$,
\[
\log \left( \frac{\langle S_t^{\, f} g_{\tau_1} , g_{\tau_2} \rangle_{\gamma_1}}{\| g_{\tau_1} \|_{L^p(\gamma_1)} \, \| g_{\tau_2} \|_{L^{q'}(\gamma_1)}}  \right) = -\frac14 \big[  (p-1) k^2 - 2 e^{-bt} k + (q' - 1)  + o(1) \big] \tau_2^2 \to +\infty.
\]
By \eqref{eq:bilinear}, we conclude the proof.

For the Laguerre operator with $\alpha\ge -1/2$, the boundedness regime follows from \eqref{eq:pre-hypercontractivity} by the same argument as in the Ornstein--Uhlenbeck case. For the blow-up regime, we continue to use the bilinear test \eqref{eq:bilinear} with the exponential family $g_\tau$.

By the computation in the proof of Proposition~\ref{prop:Laguerre-exp-growth}, for any $1 \le r < \infty$, $g_\tau \in L^r(\mu_\alpha)$ if and only if $\tau \in (0, 1/r)$. Compared with the Ornstein--Uhlenbeck case, the difficulty here is that $g_\tau$ may fail to lie in $L^2(\mu_\alpha)$ so that we cannot use Parseval's relation directly. We first prove the following more general result.

\begin{theorem}[A necessary condition for $L^p \to L^q$ Laguerre multipliers]\label{thm:parseval}
Let $\alpha > -1$ and $1 < p, q< \infty$. Let $A$ be a spectral multiplier associated with a bounded non-negative sequence $(a_n)_{n \ge 0}$, i.e., $A \ell_n^{\alpha} = a_n \, \ell_n^{\alpha}$. Assume that $A$ extends to a bounded linear operator from $L^p(\mu_\alpha)$ to $L^q(\mu_\alpha)$. Then for any $\tau_1 \in (0, 1/p)$ and $\tau_2 \in (0, 1/q')$,
\begin{equation} \label{eq:parseval-I}
\begin{split}
 \langle A g_{\tau_1}, g_{\tau_2} \rangle_{\mu_\alpha} = (1 - \tau_1)^{-(\alpha + 1)} (1 - \tau_2)^{-(\alpha + 1)} \bigg(a_0 + \sum_{n = 1}^{\infty} a_n \frac{\Gamma(n + \alpha+1)}{\Gamma(\alpha+1) n!}  \varrho(\tau_1, \tau_2)^n\bigg) < \infty,
\end{split}
\end{equation}
where
\[
\varrho(\tau_1, \tau_2) : = \frac{\tau_1 \tau_2}{(1 - \tau_1) (1 - \tau_2)} \ge 0.
\]
Consequently, the sequence $(a_n)_{n \ge 0}$ must satisfy
\begin{equation} \label{eq:radius}
\limsup_{n \to \infty} a_n^{1/n} \le \frac{p-1}{q-1}.
\end{equation}
Moreover, there exists a constant $C_{\alpha, p, q} > 0$ such that for any $n \in \mathbb N_0$,
\begin{equation} \label{eq:control}
a_n \le C_{\alpha, p, q} \, \| A \|_{p \to q} \bigl(  \frac{p-1}{q-1} \bigr)^n (n+1)^{(\alpha + 1)(p^{-1} - q^{-1})+1}.
\end{equation}
\end{theorem}

\begin{proof} [Proof of Theorem~\ref{thm:parseval}]

To establish \eqref{eq:parseval-I}, we use an approximation argument based on the  H\"ormander-type holomorphic functional calculus developed in \cite{CD}. For $1 <p, q< \infty$, define $\theta_{p, q} :=  \arcsin (\max \{ |1 - 2/p|, \, |1 - 2/q| \})$ and choose $\omega \in\bigl(1, \frac{\pi}{2\theta_{p,q}} \bigr)$. For any $\varepsilon > 0$, let $M_\varepsilon$ be the spectral multiplier with symbol $m_\varepsilon(\lambda): = \exp(-\varepsilon \lambda^\omega)$, i.e., $M_\varepsilon \ell_n^{\alpha} = e^{-\varepsilon n^\omega} \, \ell_n^{\alpha}$.

\begin{lemma} \label{lem:approximation}
Let $\alpha > -1$ and $p,q,\omega$ be as above. Then the family of spectral multipliers $(M_\varepsilon)_{\varepsilon > 0}$ has the following properties:
\begin{enumerate}
\item (\emph{Uniform boundedness}) Each $M_\varepsilon$ extends to a bounded operator on both $L^p(\mu_\alpha)$ and $L^{q'}(\mu_\alpha)$. Moreover,
\[
\sup_{\varepsilon > 0} \| M_\varepsilon \|_{p \to p} < \infty, \qquad \sup_{\varepsilon > 0} \| M_\varepsilon \|_{q' \to q'} < \infty.
\]

\item (\emph{Strong convergence}) As $\varepsilon\to0$, we have $M_\varepsilon\to I$ strongly on
$L^p(\mu_\alpha)$ and on $L^{q'}(\mu_\alpha)$; that is, for $r\in\{p, q'\}$ and any $h \in L^r(\mu_\alpha)$,
\[
\|M_\varepsilon h - h\|_{L^r(\mu_\alpha)} \to 0.
\]

\item (\emph{Regularization of the exponential family}) For any $\varepsilon > 0$ and $\tau \in \bigl( 0, \frac{1}{p \wedge q'} \bigr) $, $M_\varepsilon g_\tau \in L^2(\mu_\alpha)$. Moreover, for any $n \in \mathbb N_0$,
\begin{equation} \label{eq:fourier-coefficient}
\langle M_\varepsilon g_{\tau}, \ell_n^{\alpha} \rangle_{\mu_\alpha} = (-1)^n e^{-\varepsilon n^\omega} (1 - \tau)^{-(\alpha + 1)} \Big(\frac{\Gamma(n+\alpha+1)}{\Gamma(\alpha+1)\,n!}\Big)^{1/2}
\Big(\frac{\tau}{1- \tau}\Big)^n.
\end{equation}
\end{enumerate}
\end{lemma}

\begin{proof}
(a) Fix $\theta_{p, q} < \theta < \frac{\pi}{2 \omega}$. Then the function $m_\varepsilon(z) = \exp(-\varepsilon z^\omega)$ is a bounded holomorphic function on the sector $\big\{ z \in \mathbb C \backslash \{0\} \, ; \, |\rm{arg} \,z | < \theta \big\}$. Moreover,
\[
\sup_{\varepsilon > 0} \sup_{\lambda > 0} \big|  m_\varepsilon(e^{\pm i \theta} \lambda) \big| \le \sup_{u > 0}  e^{-\cos(\omega \theta) u} = 1.
\]
By the Cauchy theorem, the Mihlin--H\"ormander condition required in \cite[Theorem 1]{CD} holds on the smaller sector $\big\{ z \in \mathbb C \backslash \{0\} \, ; \, |\rm{arg} \,z | < \theta_{p, q} \big\}$ uniformly in $\varepsilon > 0$. Then \cite[Theorem 1]{CD} yields the desired conclusion.

(b) For polynomials, the claim is clear since $\lim_{\varepsilon \to 0} m_\varepsilon(n) = 1$ for any $n \in \mathbb N_0$. Therefore, (b) follows directly from (a) and the density of polynomials in $L^r(\mu_\alpha)$.

(c) Write $r : = p \wedge q'$. For our choice of $\tau$, $g_\tau \in L^r(\mu_\alpha)$. Since $M_\varepsilon$ is bounded on $L^r(\mu_\alpha)$ and $\ell_n^{\alpha}\in L^{r'}(\mu_\alpha)$, we may use the $L^r$--$L^{r'}$ duality; moreover, as $M_\varepsilon$ is a spectral multiplier of the self-adjoint operator $-L_\alpha$, it is symmetric on polynomials. Hence
\[
\langle M_\varepsilon g_{\tau}, \ell_n^{\alpha} \rangle_{\mu_\alpha} = \langle  g_{\tau}, M_\varepsilon \ell_n^{\alpha} \rangle_{\mu_\alpha} = e^{-\varepsilon n^\omega} \langle g_{\tau}, \ell_n^{\alpha} \rangle_{\mu_\alpha}.
\]
By \eqref{eq:fourier-expansion-Laguerre}, we have
\[
\langle g_{\tau}, \ell_n^{\alpha} \rangle_{\mu_\alpha} = (-1)^n (1 - \tau)^{-(\alpha + 1)} \Big(\frac{\Gamma(n+\alpha+1)}{\Gamma(\alpha+1)\,n!}\Big)^{1/2}
\Big(\frac{\tau}{1- \tau}\Big)^n.
\]
Substituting this identity yields \eqref{eq:fourier-coefficient}. To show that $M_\varepsilon g_\tau\in L^2(\mu_\alpha)$, it suffices to verify that
\[
\sum_{n=0}^\infty \big|\langle M_\varepsilon g_\tau,\ell_n^{\alpha}\rangle_{\mu_\alpha}\big|^2<\infty,
\]
this follows at once from \eqref{eq:fourier-coefficient} and the fact that $\omega >1$.
\end{proof}

Then we apply Lemma~\ref{lem:approximation} to prove \eqref{eq:parseval-I}. Indeed, combining the $L^p \to L^q$ boundedness of $A$ and Lemma~\ref{lem:approximation} (b), we have
\[
\langle A g_{\tau_1} , g_{\tau_2} \rangle_{\mu_\alpha} = \lim_{\varepsilon \to 0} \, \langle A M_\varepsilon g_{\tau_1} , M_\varepsilon g_{\tau_2} \rangle_{\mu_\alpha}.
\]
On the other hand, by Lemma~\ref{lem:approximation} (c), $M_\varepsilon g_{\tau_1},  M_\varepsilon g_{\tau_2} \in L^2(\mu_\alpha)$. Since $(a_n)_{n\ge0}$ is bounded, $A$ extends to a bounded operator on $L^2(\mu_\alpha)$ with $\|A\|_{2\to2}=\sup_{n\ge0} a_n$. Therefore, by Parseval's relation and \eqref{eq:fourier-coefficient},
\begin{align*}
& \langle A M_\varepsilon g_{\tau_1} , M_\varepsilon g_{\tau_2} \rangle_{\mu_\alpha}\\
& \quad = (1 - \tau_1)^{-(\alpha + 1)} (1 - \tau_2)^{-(\alpha + 1)} \bigg(a_0 + \sum_{n = 1}^{\infty} e^{-2\varepsilon n^\omega} a_n \frac{\Gamma(n + \alpha+1)}{\Gamma(\alpha+1) n!}  \varrho(\tau_1, \tau_2)^n\bigg)
\end{align*}
Letting $\varepsilon \to 0$ and using the monotone convergence theorem yields \eqref{eq:parseval-I}.

Notice that
\[
\sup_{\tau_1 \in (0, 1/p), \, \tau_2 \in (0, 1/q')} \varrho(\tau_1, \tau_2) = \frac{q - 1}{p-1}.
\]
Hence, by \eqref{eq:parseval-I} and $\Gamma(n+\alpha +1)/n! \simeq (n+1)^\alpha$, the series $\sum_{n = 0}^{\infty} a_n (n+1)^\alpha z^n$
converges for all $z \in \bigl(0, \frac{q-1}{p-1} \bigr)$, which implies \eqref{eq:radius}. Moreover, by \eqref{eq:bilinear},
\begin{equation*}
\langle A g_{\tau_1} , g_{\tau_2} \rangle_{\mu_\alpha} \le \| A \|_{p \to q} \| g_{\tau_1} \|_{L^p(\mu_\alpha)} \, \| g_{\tau_2} \|_{L^{q'}(\mu_\alpha)} = \| A \|_{p \to q} (1 - p \tau_1)^{-(\alpha + 1)/p} \, (1 - q' \tau_2)^{-(\alpha + 1)/q'}.
\end{equation*}
Combining this with \eqref{eq:parseval-I} and note that each term in \eqref{eq:parseval-I} is non-negative, we yield for any $n \in \mathbb N_0$,
\begin{equation} \label{eq:control-II}
a_n (n+1)^\alpha \varrho(\tau_1, \tau_2)^n \leq c_{p, q, \alpha}\, \| A \|_{p \to q} (1 - p \tau_1)^{-(\alpha + 1)/p} \, (1 - q' \tau_2)^{-(\alpha + 1)/q'}.
\end{equation}
We choose
\[
\tau_{1, n} := \frac{n}{p[n + (\alpha+1) (p^{-1} + {q'}^{-1})]}, \qquad \tau_{2, n} := \frac{n}{q' [n + (\alpha+1) (p^{-1} + {q'}^{-1})]}.
\]
Then a direct computation gives
\[
\varrho(\tau_{1, n}, \tau_{2, n}) = \frac{q-1}{p-1} \, \big( 1 + O(1/n) \big),
\]
where the implicit constants depend only on $\alpha, p, q$. Substituting $(\tau_1, \tau_2) = (\tau_{1, n}, \tau_{2, n})$ into \eqref{eq:control-II} gives \eqref{eq:control}. The proof is complete.
\end{proof}

We now turn to the subordinated semigroups $S_t^{\,f}$ in Theorem~\ref{thm:intro-main2}. Note that $S_t^{\,f}$ is the spectral multiplier associated with $e^{-tf(n)}$. By Lemma~\ref{lem:OU-drift-factorization},
\[
\lim_{n \to \infty} e^{-t f(n)/n} = e^{-bt},
\]
Theorem~\ref{thm:parseval} implies that, if $e^{-bt} >  (p-1)/(q-1)$, namely $q > 1+ (p-1)e^{bt}$, then $\big\| S_t^{\,f} \big\|_{p\to q} = +\infty$. This completes the proof. We finally remark that the above proof holds for all $\alpha > -1$, and taking $f(\lambda) = \sqrt{\lambda}$ yields the blow-up statement for the Poisson--Laguerre semigroup.

\subsection{Jacobi case}\label{sec:jacobi-case}

In this section we address the Jacobi case with $\alpha, \beta \ge -1/2$. By \cite[Theorem 3.3.15(2)]{W05}, the ultracontractivity of the form
\[
\big\| S_t^{\, f}   \big\|_{1 \to \infty} \lesssim (1 \wedge t)^{-\sigma/2}, \quad t > 0
\]
is equivalent to the super Poincar\'e inequality
$$
\|u\|_{L^2(\mu)}^2
\le r\,\langle f(-L)u,u\rangle_{L^2(\mu)}+ c(1 + r^{-\sigma/2})\,\| u \|_{L^1(\mu)}^2,
\qquad r>0,\ u\in D(f(-L)),
$$
for some $c > 0$. In particular, by \eqref{eq:jacobi-ultra} we have
$$
\|u\|_{L^2(\mu_{\alpha, \beta})}^2
\le r\,\langle (-L_{\alpha, \beta}) u,u\rangle_{L^2(\mu_{\alpha, \beta})}+ \beta(r) \, \| u \|_{L^1(\mu_{\alpha, \beta})}^2,
\qquad r>0,\ u\in D(L_{\alpha, \beta}),
$$
with $\beta(r) = c_{\alpha, \beta} \big( 1 + r^{-(\max\{\alpha, \beta \}+1)} \big)$. By Proposition~\ref{prop:SW-superP}, this yields
$$
\|u\|_{L^2(\mu_{\alpha, \beta})}^2
\le r\,\langle f(-L_{\alpha, \beta})u,u\rangle_{L^2(\mu_{\alpha, \beta})}+ \beta_f(r) \,\| u \|_{L^1(\mu_{\alpha, \beta})}^2,
\qquad r>0,\ u\in D(f(-L_{\alpha, \beta})),
$$
where
$$
\beta_f(r) : = 4\,\beta\!\left(\frac{1}{2\,f^{-1}(2/r)}\right).
$$

Assume $\liminf_{\lambda \to \infty} \lambda^{-\theta} f(\lambda) > 0$. Then there exist $\lambda_0, C > 0$ such that $f^{-1}(\lambda) \le C \lambda^{1/\theta}$ for any $\lambda \ge \lambda_0$. Consequently, for $r \in (0, 2\lambda_0^{-1}]$,
\[
\frac{1}{2\,f^{-1}(2/r)} \ge (2^{1+1/\theta}C)^{-1} r^{1/\theta} \simeq r^{1/\theta}.
\]
Combining this with the definition of $\beta_f$, we obtain, for any $r > 0$,
$$
\beta_f(r) \lesssim 1 + r^{- \frac{\max\{ \alpha, \beta \} + 1}{\theta}}.
$$
By \cite[Theorem~3.3.15(2)]{W05} again, this implies
\[
\big\| S_t^{\, f}   \big\|_{1 \to \infty} \lesssim (1 \wedge t)^{-\frac{(\max\{ \alpha, \beta \} + 1)}{\theta}}, \quad t > 0.
\]
The proof follows by interpolation.

\end{document}